\documentclass[final]{amsart}

\usepackage{amsmath,amsthm,amssymb}
\usepackage{amsfonts}
\usepackage[mathscr]{eucal}
\usepackage{ascmac}
\usepackage{epic, eepic}
\usepackage[dvipdfmx]{graphicx}
\usepackage{cases}
\usepackage{color}
\usepackage{url}
\usepackage{bm}
\usepackage[all]{xy}
\usepackage[top=25truemm,bottom=25truemm,left=30truemm,right=30truemm]{geometry}
\usepackage{enumerate}

\theoremstyle{plane}
\newtheorem{thm}{Theorem}[section]
\newtheorem{prop}[thm]{Proposition}
\newtheorem{lem}[thm]{Lemma}
\newtheorem{fact}[thm]{Fact}
\theoremstyle{definition}
\newtheorem{dfn}[thm]{Definition}
\theoremstyle{remark}
\newtheorem{rmk}[thm]{Remark}

\allowdisplaybreaks

\newcommand{\rank}{\operatorname{rank}}

\newcommand{\sgn}{\operatorname{sgn}}
\newcommand{\R}{\bm{R}}

\newcommand{\Sig}{\Sigma}
\newcommand{\what}{\widehat}
\newcommand{\wtilde}{\widetilde}
\renewcommand{\phi}{\varphi}
\renewcommand{\epsilon}{\varepsilon}

\numberwithin{equation}{section}

\begin{document}
\title[Principal curvatures and parallels of fronts]
{Principal curvatures and parallel surfaces of wave fronts}

\author[K. Teramoto]{Keisuke Teramoto}
\thanks{The author partly supported by the Grant-in-Aid for JSPS Fellows, No. 17J02151.}
\address{Department of Mathematics, 
Graduate School of Science, 
Kobe University, 
Rokko, Kobe 657-8501, Japan}
\email{teramoto@math.kobe-u.ac.jp}

\maketitle
\begin{abstract}
We give criteria for which a principal curvature becomes a bounded $C^\infty$-function at 
non-degenerate singular points of wave fronts by using geometric invariants.
As applications, we study singularities of parallel surfaces and extended distance squared functions of wave fronts.  Moreover, we relate these singularities to some geometric invariants of fronts. 
\end{abstract}

\section{Introduction}
Wave fronts in the Euclidean $3$-space $\R^3$ are surfaces which may have certain singularities. 
Since wave fronts have a well-defined unit normal vector even at singularities, 
they might be considered as generalizations of immersed surfaces in $\R^3$. 
Recently, there have been several studies of wave fronts from differential geometric viewpoints 
(see \cite{hhnsuy,im,krsuy,ms,msuy,nuy,suy3,suy,suy4,t1}, for example). 
In particular, the behavior of Gaussian and mean curvature of wave fronts are well investigated, 
and relations between boundedness of Gaussian curvature near non-degenerate singular points 
and geometric invariants of wave fronts are known (cf. \cite{msuy,suy}). 
For principal curvatures, Murata and Umehara \cite{mu} showed that 
at least one principal curvature is unbounded near a singular point. 
However, another principal curvature may be a bounded $C^\infty$-function. 
Hence it is natural to ask which properties of wave fronts determine boundedness of principal curvatures at singular points.

In this paper, we give an explicit criterion for which a principal curvature becomes a bouded $C^\infty$-function 
near non-degenerate singular points of wave fronts in terms of geometric invariants (Theorem \ref{princi1}). 
(This kind of criteria for the case of cuspidal edges is given in \cite[Proposition 2.2]{t1}.)
For a bounded principal curvature, we can define a principal vector with respect to it. 
On the other hand, the image of the set of non-degenerate singular points (singular locus) is a curve on a wave front.
Thus we can extend the notion of a line of curvature to a singular locus by using the principal vector and 
give a condition for the singular locus to be a line of curvature on wave fronts (Proposition \ref{curvline}).

As an application, we consider singularities of parallel surfaces on wave fronts. 
We studied parallel surfaces of cuspidal edges and gave a characterization for swallowtails appearing on parallel surfaces 
of cuspidal edges in terms of geometric properties of cuspidal edges in \cite{t1}. 
However, we have not characterized other singularities which appear on parallel surfaces of cuspidal edges or 
wave fronts, in their differential geometric contexts.  
Thus we show relations between types of singularities of parallel surfaces on wave fronts and 
geometric properties of initial wave fronts (Theorem \ref{sing_para}). 
To characterize singularities, the notion of ridge points for wave fronts will play important roles. 
(Ridge points for regular surfaces are introduced by Porteous \cite{p1}.) 
In addition, we consider constant principal curvature (CPC) lines near cuspidal edges. 
It is known that CPC lines correspond to the set of singular points of parallel surfaces (\cite{fh1,fh2}). 
Using parallel surfaces, we define special points (landmark in the sense of \cite{p2}) on cuspidal edge as 
cusps of CPC lines, which seems not to have appeared in the literature (Subsection \ref{cpc}). 

Finally, we study the extended distance squared function on wave fronts. 
For the case of generic regular surfaces, singularities of extended distance squared functions correspond to 
types of singularities of parallel surfaces (cf. \cite[Theorem 3.4]{fh1}).
However, for wave fronts, the same statement does not hold, 
in fact, different kinds of singularities ($D$-type) will appear (Theorem \ref{D4front}).  

All maps and functions considered here are of class $C^\infty$ unless otherwise stated. 


\section{Preliminaries}\label{prelim}
\subsection{Wave fronts}
We recall some properties of wave fronts. 
For details, see \cite{agv,fsuy,ifrt,msuy,suy}. 

A map $f:\Sig\to\R^3$ is called a {\it wave front} (or a {\it front}) if 
there exists a unit normal vector $\nu$ to $f$ such that the pair $L_f=(f,\nu):\Sig\to\R^3\times S^2$ gives an immersion, 
where $\Sig\subset(\R^2;u,v)$ is a domain and $S^2$ denotes the unit sphere in $\R^3$ (cf. \cite{agv,krsuy,suy}). 
A map $f:\Sig\to\R^3$ is called a {\it frontal} if just a unit normal vector $\nu$ to $f$ exists. 
A point $p$ is said to be a {\it singular point} of $f$ if $f$ is not an immersion at $p$. 
We denote by $S(f)$ the set of singular points of $f$.

For a frontal $f$, the function $\lambda:\Sig\to\R$ as 
$$\lambda(u,v)=\det(f_u,f_v,\nu)(u,v)\quad (f_u=\partial f/\partial u,\ f_v=\partial f/\partial v)$$
is called the {\it signed area density function} (cf. \cite{krsuy,suy}). 
By the definition of $\lambda$, $S(f)=\lambda^{-1}(0)$ holds. 
We call $p\in S(f)$ {\it non-degenerate} if $d\lambda(p)\neq0$. 
Let $p$ be non-degenerate. 
Then there exist a neighbourhood $V$ of $p$ and a regular curve 
$\gamma:(-\epsilon,\epsilon)\to V$ with $\gamma(0)=p$ such that 
$S(f)\cap V$ is locally parametrized by $\gamma$. 
Moreover, there exists a vector field $\eta$ such that $df(\eta)=\bm{0}$ along $\gamma$. 
We call $\gamma$ and $\eta$ the {\it singular curve} and the {\it null vector field}, respectively. 
Moreover, we call the image of the singular curve $\what{\gamma}=f\circ\gamma$ the {\it singular locus}. 

A non-degenerate singular point $p$ is said to be of the \textit{first kind} if $\eta(0)$ is transverse 
to $\gamma'(0)$, that is, $\det(\gamma',\eta)(0)\neq0$. 
Otherwise, it is said to be of the \textit{second kind} (\cite{msuy}). 
Moreover, we call a non-degenerate singular point of the second kind {\it admissible} if 
the singular curve consists of points of the first kind except at $p$. 
Otherwise, we call $p$ {\it non-admissible}. 

\begin{dfn}
Let $f:(\Sig,p)\to(\R^3,f(p))$ be a map-germ around $p$. 
Then $f$ at $p$ is a \textit{cuspidal edge} if the map-germ $f$ is $\mathcal{A}$-equavalent to the map-germ 
$(u,v)\mapsto(u,v^2,v^3)$ at $\bm{0}$, 
$f$ at $p$ is a \textit{swallowtail} if the map-germ $f$ is $\mathcal{A}$-equivalent to the map-germ 
$(u,v)\mapsto(u,3v^4+uv^2,4v^3+2uv)$ at $\bm{0}$, 
$f$ at $p$ is a \textit{cuspidal butterfly} if the map-germ $f$ is $\mathcal{A}$-equivalent to the map-germ 
$(u,v)\mapsto(u,4v^5+uv^2,5v^4+2uv)$ at $\bm{0}$, 
$f$ at $p$ is a \textit{cuspidal lips} if the map-germ $f$ is $\mathcal{A}$-equivalent to the map-germ 
$(u,v)\mapsto(u,3v^4+2u^2v^2,v^3+u^2v)$ at $\bm{0}$, 
$f$ at $p$ is a \textit{cuspidal beaks} if the map-germ $f$ is $\mathcal{A}$-equivalent to the map-germ 
$(u,v)\mapsto(u,3v^4-2u^2v^2,v^3-u^2v)$ at $\bm{0}$ 
and $f$ at $p$ is a {\it $D_4^+$ singularity} (resp. {\it $D_4^-$ singularity}) 
if the map-germ $f$ is $\mathcal{A}$-equivalent to 
$(u,v)\mapsto(uv,u^2+3v^2,u^2v+ v^3)$ (resp. $(u,v)\mapsto(uv,u^2-3v^2,u^2v-v^3)$) at $\bm{0}$, 
where two map-germs $f,\,g:(\R^2,\bm{0})\to(\R^3,\bm{0})$ are \textit{$\mathcal{A}$-equivalent} 
if there exist diffeomorphism-germs $\theta:(\R^2,\bm{0})\to(\R^2,\bm{0})$ on the source and 
$\Theta:(\R^3,\bm{0})\to(\R^3,\bm{0})$ on the target such that $\Theta\circ f=g\circ\theta$ holds.
\end{dfn}
We note that generic singularities of fronts are cuspidal edges and swallowtails 
and generic singularities of one-parameter bifurcation of fronts are 
cuspidal lips/beaks, cuspidal butterflies and $D_4^\pm$ singularities in addition to above two (see \cite{agv,ifrt}).

\begin{rmk}
Cuspidal edges are non-degenerate singular points of the first kind. 
On the other hand, swallowtails and cuspidal butterflies are 
of the admissible second kind (cf. \cite{msuy}). 
Thus generic singularities of fronts are admissible. 
\end{rmk}

\begin{fact}[\cite{is,ist,krsuy,suy3}]\label{crit_front}
Let $f:(\Sig,p)\to\R^3$ be a front germ, $\nu$ a unit normal to $f$ and $p$ a corank one singular point, 
namely, $\rank{df_p}=1$.
\begin{enumerate}
\item[{\rm(1)}] Suppose that $p$ is a non-degenerate singular point. 
\begin{itemize}
\item $f$ at $p$ is $\mathcal{A}$-equivalent to a cuspidal edge if and only if $\eta\lambda(p)\neq0$.
\item $f$ at $p$ is $\mathcal{A}$-equivalent to a swallowtail if and only if $\eta\lambda(p)=0$ 
and $\eta\eta\lambda(p)\neq0$.
\item $f$ at $p$ is $\mathcal{A}$-equivalent to a cuspidal butterfly 
if and only if $\eta\lambda(p)=\eta\eta\lambda(p)=0$ and $\eta\eta\eta\lambda(p)\neq0$.
\end{itemize}
\item[{\rm(2)}] Suppose that $p$ is a degenerate singular point.
\begin{itemize}
\item $f$ at $p$ is $\mathcal{A}$-equivalent to a cuspidal lips if and only if $\det\mathcal{H}_\lambda(p)>0$.
\item $f$ at $p$ is $\mathcal{A}$-equivalent to a cuspidal beaks if and only if 
$\eta\eta\lambda(p)\neq0$ and $\det\mathcal{H}_\lambda(p)<0$.
\end{itemize}
\end{enumerate}
Here $\lambda$ is the signed area density function, $\eta$ the null vector field 
and $\mathcal{H}_\lambda$ the Hessian matrix of $\lambda$.
\end{fact}
We note that there is a criterion for a {\it cuspidal cross cap} 
which appears on a frontal surface defined as a map-germ $\mathcal{A}$-equivalent to 
$(u,v)\mapsto(u,v^2,uv^3)$ at $\bm{0}$ (\cite[Theorem 1.4]{fsuy}).  

We recall behavior of curvatures of fronts near non-degenerate singular points $p$. 
Let $f:\Sig\to\R^3$ be a front and $\nu$ a unit normal vector. 
Let $K$ and $H$ denote the Gaussian and the mean curvature of a front $f$. 
It is known that $H$ is unbounded near $p$ ({\cite[Corollary 3.5]{suy}}). 
On the other hand, for the Gaussian curvature $K$, it is known that 
$K$ is bounded near $p$ if and only if the second fundamental form 
vanishes along the singular curve $\gamma$ ({\cite[Theorem 3.1]{suy}}). 

Next we recall behavior of principal curvature maps of a front $f$ at singular points.
Let us assume that there are no umbilic points on $V$.
Then there exists a local coordinate system $(U;u,v)$ centered at $p$ 
such that $f_u$ and $\nu_u$ $($resp. $f_v$ and $\nu_v)$ are linearly dependent on $U$. 
In particular, the pair $\{f_u,\nu_u\}$ (resp. $\{f_v,\nu_v\}$) 
does not vanish at the same time ({\cite[Lemma 1.3]{mu}}). 
Such a coordinate system is called a \textit{principal curvature line coordinate} introduced in \cite{mu}. 
For this local coordinate system $(U;u,v)$, we define the maps $\Lambda_i:U\to P^1(\R)$ $(i=1,2)$ 
which are called the {\it principal curvature maps} (\cite{mu}) as 
the proportional ratio of the real projective line $P^1(\R)$ by
\begin{equation}\label{pmap}
\Lambda_1=[-\nu_u:f_u],\quad\Lambda_2=[-\nu_v:f_v].
\end{equation} 
\begin{fact}[{\cite[Lemma 1.7]{mu}}]\label{bhv_pr_map}
Let $f:\Sig\to\R^3$ be a front and $\Lambda_1,\Lambda_2$ be the principal curvature maps. 
Then $p\in \Sig$ is a singular point if and only if either $\Lambda_1(p)=[1:0]$ or 
$\Lambda_2(p)=[1:0]$ holds.
\end{fact}
By Fact \ref{bhv_pr_map}, one principal curvature function of a wave front is bounded 
and the other is unbounded near a singular point. 

\subsection{Invariants of a cuspidal edge}
Let $f:\Sig\to\R^3$ be a frontal, $p\in \Sig$ a non-degenerate singular point and $\nu$ a unit normal vector. 
Then we can take the following local coordinate system around $p$. 
\begin{dfn}[\cite{krsuy,msuy,suy}]\label{adapt_1}
A local coordinate system $(U;u,v)$ centered at a singular point of the first kind 
(resp. of the second kind) $p$ is called 
\textit{adapted} if it is compatible with the orientation of $\Sig$ and satisfies the following conditions:\\
(1) the $u$-axis is the singular curve,\\
(2) $\eta=\partial_v$ (resp. $\eta=\partial_u+\epsilon(u)\partial_v$ with $\epsilon(0)=0$) 
gives the null vector field on the $u$-axis,\\ 
(3) there are no singular points other than the $u$-axis.
\end{dfn}
Let $p$ be a cuspidal edge and $(U;u,v)$ an adapted coordinate system centered at $p$.
Since $df(\eta)=f_v=\bm{0}$ along the $u$-axis, there exists a map $\phi:U\to\R^3\setminus\{\bm{0}\}$ 
such that $f_v=v\phi$. 
We note that $f_{vv}=\phi$ holds along the $u$-axis. 
Since $\eta\lambda=\det(f_u,\phi,\nu)\neq0$ on the $u$-axis by Fact \ref{crit_front}, 
the pair $\{f_u,\phi,\nu\}$ gives a frame (cf. \cite{msuy,t1}). 
\begin{lem}[{\cite[Lemma 2.1]{t1}}]\label{wein1}
It holds that
\begin{equation*}
\nu_u=\frac{\wtilde{F}\wtilde{M}-\wtilde{G}\wtilde{L}}{\wtilde{E}\wtilde{G}-\wtilde{F}^2}f_u+
\frac{\wtilde{F}\wtilde{L}-\wtilde{E}\wtilde{M}}{\wtilde{E}\wtilde{G}-\wtilde{F}^2}\phi,\ 
\nu_v=\frac{\wtilde{F}\wtilde{N}-v\wtilde{G}\wtilde{M}}{\wtilde{E}\wtilde{G}-\wtilde{F}^2}f_u+
\frac{v\wtilde{F}\wtilde{M}-\wtilde{E}\wtilde{N}}{\wtilde{E}\wtilde{G}-\wtilde{F}^2}\phi,
\end{equation*}
where $\wtilde{E}=\|f_u\|^2$, $\wtilde{F}=\langle{f_u,\phi}\rangle$, $\wtilde{G}=\|\phi\|^2$, 
$\wtilde{L}=-\langle{f_u,\nu_u}\rangle$, $\wtilde{M}=-\langle{\phi,\nu_u}\rangle$ and $\wtilde{N}=-\langle{\phi,\nu_v}\rangle$.
\end{lem}

For cuspidal edges, several geometric invariants are studied (for example, see \cite{ms,msuy,nuy,suy,suy4,su}). 
By using an adapted coordinate system $(U;u,v)$ and the frame $\{f_u,\phi,\nu\}$, 
we set the following invariants along the $u$-axis: 
\begin{align*} 
\kappa_s(u)&=\sgn(\lambda_v)\frac{\det(f_u,f_{uu},\nu)}{\|f_u\|^3}(u,0),\ 
\kappa_\nu(u)=\frac{\langle f_{uu},\nu\rangle}{\|f_u\|^2}(u,0),\ 
\kappa_c(u)=\frac{\|f_u\|^{3/2}\det(f_u,\phi,f_{vvv})}{\|f_u\times \phi\|^{5/2}}(u,0),\\
\kappa_t(u)&=\frac{\det(f_u,\phi,f_{uvv})}{\|f_u\times \phi\|^2}(u,0)-
\frac{\det(f_u,\phi,f_{uu})\langle f_u,\phi\rangle}{\|f_u\|^2\|f_u\times \phi\|^2}(u,0).
\end{align*}
$\kappa_s$, $\kappa_\nu$, $\kappa_c$ and $\kappa_t$ 
are called the {\it singular curvature}, the {\it limiting normal curvature}, 
the {\it cuspidal curvature} and the {\it cusp-directional torsion}, respectively. 
See \cite{hhnsuy,ms,msuy,suy,su} for details of their geometric meanings. 
We note that these invariants can be defined on frontals with singular points of the first kind, 
and for $\kappa_\nu$, we can define it at singular points of the second kind (cf. \cite[(1.2)]{msuy}). 
\begin{lem}\label{cuspidalcurv}
Under the above settings, $\kappa_\nu$, $\kappa_c$ and $\kappa_t$ 
can be expressed as 
\begin{equation}
\kappa_\nu(u)=\frac{\wtilde{L}}{\wtilde{E}}(u,0),\quad
\kappa_c(u)
=\pm\frac{2\wtilde{E}^{3/4}\wtilde{N}}{(\wtilde{E}\wtilde{G}-\wtilde{F}^2)^{3/4}}(u,0),\quad
\kappa_t(u)=\pm\frac{\wtilde{E}\wtilde{M}-\wtilde{F}\wtilde{L}}{\wtilde{E}\sqrt{\wtilde{E}\wtilde{G}-\wtilde{F}^2}}(u,0) \label{cusp2}
\end{equation}
along the $u$-axis, where $\pm$ depends on the orientation of the frame $\{f_u,\phi,\nu\}$. 
\end{lem}
\begin{proof}
One can check that $\kappa_\nu$ can be expressed as above by defintions of functions.
We show $\kappa_c$ and $\kappa_t$ can be written as the above formulas. 
Since $\nu$ is perpendicular to both $f_u$ and $\phi$, $\nu$ can be written as 
$\nu=\pm(f_u\times\phi)/\|f_u\times\phi\|$. 

First, we show that $\kappa_c$ can be written as above. 
We note that $f_{vvv}=2\phi_v$ holds on the $u$-axis. 
Since $\wtilde{N}=-\langle\phi,\nu_v\rangle=\langle\phi_v,\nu\rangle$, 
$\kappa_c$ on the $u$-axis is expressed as 
$$\kappa_c(u)=\frac{2\wtilde{E}^{3/4}\det(f_u,\phi,\phi_v)}{\|f_u\times\phi\|^{5/2}}(u,0)
=\pm\frac{2\wtilde{E}^{3/4}\langle\nu,\phi_v\rangle}{\|f_u\times\phi\|^{3/2}}(u,0)=
\pm\frac{2\wtilde{E}^{3/4}\wtilde{N}}{(\wtilde{E}\wtilde{G}-\wtilde{F}^2)^{3/4}}(u,0)$$
on the $u$-axis. 

Next, we consider $\kappa_t$. 
Since $f_{uvv}=\phi_u$ and $\langle\phi_u,\nu\rangle=-\langle\phi,\nu_u\rangle=\wtilde{M}$ on the $u$-axis, we see that 
$$\kappa_t(u)=\frac{\det(f_u,\phi,\phi_u)}{\wtilde{E}\wtilde{G}-\wtilde{F}^2}(u,0)-
\frac{\det(f_u,\phi,f_{uu})\wtilde{F}}{\wtilde{E}(\wtilde{E}\wtilde{G}-\wtilde{F}^2)}(u,0)
=\pm\frac{\wtilde{E}\wtilde{M}-\wtilde{F}\wtilde{L}}{\wtilde{E}\sqrt{\wtilde{E}\wtilde{G}-\wtilde{F}^2}}(u,0).$$
\end{proof}

It is known that $\kappa_c(p)$ does not vanish if $p$ is a cuspidal edge (cf. {\cite[Lemma 2.11]{msuy}}).
In particular, $\wtilde{N}$ never vanishes on the $u$-axis by Lemma \ref{cuspidalcurv}. 
Take an adapted coordinate system $(U;u,v)$ with $\eta\lambda(u,0)>0$. 
Then $\sgn(\kappa_c)=\sgn(\wtilde{N})$ holds on the $u$-axis (see Lemma \ref{cuspidalcurv}). 
If $\eta\lambda(u,0)<0$, $\sgn(\kappa_c)=-\sgn(\wtilde{N})$ holds. 

We define the following functions on $U\setminus\{v=0\}$ as
\begin{equation}
\kappa_+=\frac{2(\wtilde{L}\wtilde{N}-v\wtilde{M}^2)}{\wtilde{A}+\wtilde{B}},\quad
\kappa_-=\frac{2(\wtilde{L}\wtilde{N}-v\wtilde{M}^2)}{\wtilde{A}-\wtilde{B}},\label{k1}
\end{equation}
where 
$\wtilde{A}=\wtilde{E}\wtilde{N}-2v\wtilde{F}\wtilde{M}+v\wtilde{G}\wtilde{L},\,
\wtilde{B}=\sqrt{\wtilde{A}^2-4v(\wtilde{E}\wtilde{G}-\wtilde{F}^2)(\wtilde{L}\wtilde{N}-v\wtilde{M}^2)}$. 
(The reason why $\kappa_\pm$ can be defined as these forms is found in \cite[page 55]{t1}.)
These functions are well-defined on $U\setminus\{v=0\}$. 
We remark that $\kappa_+$ (resp. $\kappa_-$) becomes $-\kappa_-$ (resp. $-\kappa_+$) 
if we change $\nu$ to $-\nu$. 
Let $K$ and $H$ be the Gaussian and the mean curvature of $f$ defined on $U\setminus\{v=0\}$. 
Then $K=\kappa_+\kappa_-$ and $2H=\kappa_+ +\kappa_-$ hold. 
Thus we may treat $\kappa_+$ and $\kappa_-$ as {\it principal curvatures} of $f$ defined on $U\setminus\{v=0\}$. 
Here $K$ and $H$ can be expressed as 
$$K=\frac{\wtilde{L}\wtilde{N}-v\wtilde{M}^2}{v(\wtilde{E}\wtilde{G}-\wtilde{F}^2)},\quad
H=\frac{\wtilde{E}\wtilde{N}-2v\wtilde{F}\wtilde{M}+v\wtilde{G}\wtilde{L}}{2v(\wtilde{E}\wtilde{G}-\wtilde{F}^2)}$$
on the set of regular points. 
We note that $\kappa_\pm=H\mp\sqrt{H^2-K}$ hold on the set of regular points. If we take a principal curvature line coordinate (\cite{mu}), then fractional expressions of principal curvature maps 
$\Lambda_i$ $(i=1,2)$ as in \eqref{pmap} coincide with principal curvatures $\kappa_\pm$.
\subsection{Invariants of a singular point of the second kind} 
Let $f:\Sig\to\R^3$ be a frontal, $p$ a non-degenerate singular point of the second kind 
and $\nu$ a unit normal vector to $f$. 
We fix an adapted coordinate system $(U;u,v)$ in the following (see Definition \ref{adapt_1}). 
Taking a null vector field $\eta$, 
there exists a function $\epsilon=\epsilon(u)$ on the $u$-axis with $\epsilon(0)=0$ 
so that $\eta=\partial_u+\epsilon(u)\partial_v$ (see \cite{msuy}). 
(We note that if $p$ is non-admissible, $\epsilon\equiv0$ holds on the $u$-axis, namely, $\eta=\partial_u$.) 
Thus it follows that $df(\eta)=f_u+\epsilon(u)f_v=\bm{0}$ holds along the $u$-axis. 
On the other hand, since the $u$-axis gives the singular curve, 
there exists a $C^\infty$-function $\phi:U\to\R^3\setminus\{\bm{0}\}$ such that $df(\eta)=v\phi$. 
Hence we have $f_u=v\phi-\epsilon f_v$. 
We remark that $\phi,\,f_v$ and $\nu$ are linearly independent 
since $d\lambda=\det(\phi,f_v,\nu)dv\neq0$ holds on the $u$-axis. 
\begin{lem}\label{weinsec}
Under the adapted coordinate system $(U;u,v)$, $\nu_u$ and $\nu_v$ on $U$ can be written as
\begin{equation*}
\nu_u=\frac{\what{F}(v\what{M}-\epsilon\what{N})-\what{G}\what{L}}{\what{E}\what{G}-\what{F}^2}\phi+
\frac{\what{F}\what{L}-\what{E}(v\what{M}-\epsilon\what{N})}{\what{E}\what{G}-\what{F}^2}f_v,\quad
\nu_v=\frac{\what{F}\what{N}-\what{G}\what{M}}{\what{E}\what{G}-\what{F}^2}\phi+
\frac{\what{F}\what{M}-\what{E}\what{N}}{\what{E}\what{G}-\what{F}^2}f_v,
\end{equation*}
where $\what{E}=\|\phi\|^2$, $\what{F}=\langle\phi,f_v \rangle$, $\what{G}=\| f_v\|^2$, 
$\what{L}=-\langle \phi,\nu_u\rangle$, $\what{M}=-\langle\phi,\nu_v\rangle$ and $\what{N}=-\langle f_v,\nu_v\rangle$. 
\end{lem}

We now define two $C^\infty$-functions on $U\setminus\{v=0\}$ by 
\begin{equation}
\kappa_+=\frac{2((\what{L}+\epsilon(u)\what{M})\what{N}-v\what{M}^2)}{\what{A}+\what{B}},\quad
\kappa_-=\frac{2((\what{L}+\epsilon(u)\what{M})\what{N}-v\what{M}^2)}{\what{A}-\what{B}},\label{prinsec}
\end{equation}
where 
\begin{align*}
\what{A}&=\what{G}(\what{L}+\epsilon(u)\what{M})-2v\what{F}\what{M}+v\what{E}\what{N},\\
\what{B}&=\sqrt{\what{A}^2-4v(\what{E}\what{G}-\what{F}^2)\left((\what{L}+\epsilon(u)\what{M})\what{N}
-v\what{M}^2\right)}\label{sec_B}.
\end{align*}
Since the Gaussian curvature $K$ and the mean curvature $H$ of $f$ satisfy 
$K=\kappa_+\kappa_-$ and $2H=\kappa_++\kappa_-$, 
we may regard $\kappa_\pm$ as {\it principal curvatures} of $f$ on $U\setminus\{v=0\}$, 
where $K$ and $H$ are written as 
$$K=\frac{(\what{L}+\epsilon(u)\what{M})\what{N}-v\what{M}^2}{v(\what{E}\what{G}-\what{F}^2)},\quad
H=\frac{\what{G}(\what{L}+\epsilon(u)\what{M})-2v\what{F}\what{M}+v\what{E}\what{N}}{2v(\what{E}\what{G}-\what{F}^2)}$$
on $U\setminus\{v=0\}$. 
We remark that $\kappa_\pm=H\mp\sqrt{H^2-K}$ hold on the set of regular points. 

We put $\what{H}=vH$. 
This is a $C^\infty$-function on $U$. 
It follows that 
\begin{equation}
2\what{H}=\frac{\what{G}(\what{L}+\epsilon(u)\what{M})}{\what{E}\what{G}-\what{F}^2}\label{hat_h}
\end{equation}
holds along the $u$-axis (cf. \cite{msuy}). 
We note that $\what{L}+\epsilon(u)\what{M}=-\langle\phi,\eta\nu\rangle$ holds. 
It is known that $2\what{H}$ does not vanish on the $u$-axis if and only if $f$ is a front (\cite[Proposition 3.2]{msuy}).
We set 
$$\mu_c(p)=2\what{H}(p)\left(=\frac{\what{G}(p)\what{L}(p)}{\|\phi(p)\times f_v(p)\|^2}\right).$$ 
This is a geometric invariant called the {\it normalized cuspidal curvature} defined in \cite{msuy}. 
By \eqref{hat_h} and the definition of $\mu_c(p)$, 
we see that $\sgn(\mu_c(p))=\sgn(\what{L}(p))$ and $\what{L}(p)\neq0$ hold if $f$ is a front.
\begin{lem}\label{limnormalsec}
Under the above conditions, the limiting normal curvature $\kappa_\nu$ can be written as 
$\kappa_\nu=\what{N}/\what{G}$ at $p$ if $p$ is of the admissible second kind.
\end{lem}
\begin{proof}
By \cite[Proposition 1.9]{msuy}, 
$f_u=v\phi-\epsilon(u)f_v$, $f_{uu}=v\phi_u-\epsilon'(u)f_v-\epsilon(u)f_{uv}$ and 
$f_{uv}=\phi+v\phi_v-\epsilon(u)f_{vv}$, we get the conclusion.
\end{proof}

\section{Principal curvatures, principal vectors and ridge points}\label{pf}
\subsection{Boundedness of a principal curvature}
In this subsection, 
we consider boundedness of principal curvatures of fronts by using the above arguments. 
\begin{thm}\label{princi1}
Let $f:\Sig\to\R^3$ be a front and $p$ a non-degenerate singular point. 
\begin{itemize}
\item[{\rm(1)}]
Let $p$ be a cuspidal edge. 
If $\eta\lambda(p)\kappa_c(p)>0$, 
then the principal curvature $\kappa_+$ is a bounded $C^\infty$-function at $p$. 
Moreover, $\kappa_+(p)=\kappa_\nu(p)$.
\item[{\rm(2)}]
Let $p$ be of the second kind.  
If $\mu_c(p)>0$, then the principal curvature $\kappa_+$ is a bounded $C^\infty$-function at $p$. 
Moreover, $\kappa_+(p)=\kappa_\nu(p)$ if $p$ is an admissible.
\end{itemize}
Converses are also true. 
Moreover, if one of $\kappa_\pm$ is bounded at $p$, then the another is unbounded.
\end{thm}

\begin{proof}
We prove the first asserion. 
Let $f:\Sig\to\R^3$ be a front and $p$ a cuspidal edge. 
Take an adapted coordinate system $(U;u,v)$ centered at $p$. 
We show the case of $\eta\lambda(u,0)>0$. 
In this case, $\sgn(\kappa_c)=\sgn(\wtilde{N})$ holds along the $u$-axis. 
For the case of $\eta\lambda(u,0)<0$, one can show similarly. 

We now assume that $\kappa_c(p)>0$. 
Then $\wtilde{N}(p)>0$ by \eqref{cusp2}. 
Since $\wtilde{A}\pm\wtilde{B}=\wtilde{E}(\wtilde{N}\pm|\wtilde{N}|)$ and \eqref{k1}, 
we see that $\kappa_+$ is a bounded $C^\infty$-function on $U$ and 
$\kappa_+=\wtilde{L}/\wtilde{E}=\kappa_\nu$ holds at $p$. 
Conversely, we assume that the principal curvature $\kappa_+$ is a bounded $C^\infty$-function near $p$. 
In this case, it follows that $\wtilde{N}=-\langle\phi,\eta\nu\rangle$ is positive along the $u$-axis. 
This implies that $\eta\lambda\cdot\kappa_c$ is positive along the $u$-axis by \eqref{cusp2}. 
Unboundedness of $\kappa_-$ near $p$ follows from the fact that the mean curvature is unbounded near $p$. 

Next, we prove the second assertion. 
Take an adapted coordinate system $(U;u,v)$ centered at a non-degenerate singular point of the second kind $p$. 
Suppose that $\mu_c(p)=2\what{H}(p)>0$. 
It follows that $-\langle\phi,\eta\nu\rangle>0$ holds near $p$ from \eqref{hat_h}. 
Since
$\what{A}=\what{G}(-\langle\phi,\eta\nu\rangle)$, $\what{B}=|\what{A}|$ and 
$-\langle\phi,\eta\nu\rangle>0$ along the $u$-axis, 
it follows that $\what{A}+\what{B}=2\what{G}(-\langle\phi,\eta\nu\rangle)\neq0$ 
and $A-B=0$ hold on the $u$-axis. 
Hence by \eqref{prinsec}, we have $\kappa_+=\what{N}/\what{G}$ along the $u$-axis, and $\kappa_+$ is  
a bounded $C^\infty$-function. 
By Lemma \ref{limnormalsec}, we see that $\kappa_+=\kappa_\nu$ at $p$ if $p$ is admissible. 
The converse and unboundedness can be shown by using similar arguments to the first assertion. 
\end{proof}

\begin{rmk}\label{rmk:hatkappa}
We assume that $\kappa_+$ is bounded near non-degenerate singular point $p$. 
Although $\kappa_-$ is unbounded near $p$, $\lambda\kappa_-$ is bounded near $p$. 
In fact, $\kappa_-$ can be rewritten as 
$$\kappa_-=
\begin{cases}
\frac{\wtilde{A}+\wtilde{B}}{2v(\wtilde{E}\wtilde{G}-\wtilde{F}^2)}\quad (p:\ \textrm{cuspidal edge})\\
\frac{\what{A}+\what{B}}{2v(\what{E}\what{G}-\what{F}^2)} \quad (p:\ \textrm{second kind})
\end{cases}
$$
on $U\setminus\{v=0\}$ (cf. \cite{t1}). 
Thus $\lambda\kappa_-$ is written as 
$$\lambda\kappa_-=
\begin{cases}
\frac{\wtilde{A}+\wtilde{B}}{2\sqrt{\wtilde{E}\wtilde{G}-\wtilde{F}^2}} \quad (p:\ \textrm{cuspidal edge})\\
\frac{\what{A}+\what{B}}{2\sqrt{\what{E}\what{G}-\what{F}^2}}\quad (p:\ \textrm{second kind}).
\end{cases}
$$
In particular, $\lambda(p)\kappa_-(p)$ is proportional to $\kappa_c(p)$ when $p$ is a cuspidal edge, 
and $\lambda(p)\kappa_-(p)$ is proportional to $\mu_c(p)$ when $p$ is of the second kind. 
Thus $\lambda(p)\kappa_-(p)$ does not vanish.
\end{rmk}

\subsection{Principal vectors and ridge points}
By Theorem \ref{princi1}, one of $\kappa_\pm$ of fronts can be defined 
as a bounded $C^\infty$-function near non-degenerate singular points. 
This implies there is a principal vector with respect to such a principal curvature at the singular point. 
Hence we consider explicit representation of the principal vector under an adapted coordinate system. 

Let $f:\Sig\to\R^3$ be a front, $p$ a singular point of the second kind 
and $\nu$ a unit normal vector to $f$. 
Then we take an adapted coordinate system $(U;u,v)$ around $p$. 
Assume that $\mu_c(p)>0$, namely, 
$\kappa_+$ is a bounded $C^\infty$-function near $p$ in the following. 
We investigate the principal vector relative to $\kappa_+$. 

Let $I$ and $II$ denote the first and the second fundamental matrices given by 
$$I=\begin{pmatrix}
\langle f_u,f_u\rangle & \langle f_u,f_v\rangle\\
\langle f_u,f_v\rangle & \langle f_v,f_v\rangle
\end{pmatrix},\quad
II=\begin{pmatrix}
-\langle f_u,\nu_u\rangle & -\langle f_u,\nu_v\rangle\\
-\langle f_v,\nu_u\rangle & -\langle f_v,\nu_v\rangle
\end{pmatrix}.
$$
The principal vector $\bm{v}=(v_1,v_2)$ with respect to $\kappa_+$ is a never vanishing vector satisfying 
$(II-\kappa_+I)\bm{v}=\bm{0}$. 
We can write this equation as 
\begin{equation}\label{p_dir_sw}
\begin{pmatrix}
v\{\what{L}-\kappa_+(v\what{E}-\epsilon\what{F})\} & 
v(\what{M}-\kappa_+\what{F})\\
v(\what{M}-\kappa_+\what{F})-\epsilon(\what{N}-\kappa_+\what{G}) & \what{N}-\kappa_+\what{G}
\end{pmatrix}
\begin{pmatrix}
v_1 \\ v_2
\end{pmatrix}
=
\begin{pmatrix}
0\\0
\end{pmatrix}.
\end{equation}
We note that $\what{L}$ does not vanish at $p$. 
Thus we can take the principal vector $\bm{v}$ as 
\begin{equation}\label{pdir_sec}
\bm{v}=(-\what{M}+\kappa_+\what{F},
\what{L}-\kappa_+(v\what{E}-\epsilon\what{F})),
\end{equation}
by factoring out $v$ from \eqref{p_dir_sw}. 
For the case of cuspidal edges, the principal vector $\bm{v}$ with respect to $\kappa_+$ is given as follows \cite{t1}: 
\begin{equation}
\bm{v}=(\wtilde{N}-v\kappa_+\wtilde{G},-\wtilde{M}+\kappa_+\wtilde{F}).\label{pdircusp}
\end{equation}

We can extend the notion of a line of curvature as follows. 
The singular locus $\what{\gamma}=f\circ\gamma$ is 
a {\it line of curvature} if the principal vector $\bm{v}$ is tangent to $\gamma$. 
\begin{prop}\label{curvline}
Let $f:\Sig\to\R^3$ be a front, $p$ a non-degenerate singular point and $\gamma$ the singular curve passing through $p$. 
Then the following assertions hold$:$
\begin{enumerate}
\item[{\rm(1)}] Suppose that $p$ is a cuspidal edge. 
Then $\what{\gamma}$ is a line of curvature of $f$ if and only if $\kappa_t$ vanishes identically along $\gamma$.
\item[{\rm(2)}] Suppose that $p$ is of the second kind. 
Then $\what{\gamma}$ can not be a line of curvature.
\end{enumerate}
\end{prop}
\begin{proof}
First, we show assertion (1). 
Take an adapted coordinate system $(U;u,v)$ centered at a cuspidal edge $p$ satisfying $\eta\lambda(u,0)>0$. 
Assume that $\kappa_+$ is bounded on $U$. 
Then the principal vector $\bm{v}=(v_1,v_2)$ relative to $\kappa_+$ is given by \eqref{pdircusp}. 
Since $\kappa_+=\wtilde{L}/\wtilde{E}$ on the $u$-axis, 
$v_2$ can be written as 
$$v_2=-\wtilde{M}+\kappa_+\wtilde{F}=-\frac{\wtilde{E}\wtilde{M}-\wtilde{F}\wtilde{L}}{\wtilde{E}}=
-\kappa_t\sqrt{\wtilde{E}\wtilde{G}-\wtilde{F}^2}$$
along the $u$-axis by Lemma \ref{cuspidalcurv}. 
Thus $v_2$ vanishes on the $u$-axis if and only if $\kappa_t$ vanishes along the $u$-axis, 
and we get the conclusion. 

Next, we show (2). 
Take an adapted coordinate system $(U;u,v)$ around $p$ and assume that $\mu_c(p)>0$ holds. 
In this case, $\kappa_+$ is bounded on $U$ and the principal vector 
$\bm{v}=(v_1,v_2)$ of $\kappa_+$ is given as \eqref{pdir_sec}. 
The second component $v_2$ is written as 
$$v_2=\what{L}+\epsilon\kappa_+\what{F}$$
along the $u$-axis. 
Thus we have $v_2=\what{L}\neq0$ at $p$. 
This implies that the $u$-axis can not be the line of curvature. 
\end{proof}

Using the principal curvature $\kappa_+$ and the principal vector $\bm{v}$ relative to $\kappa_+$, 
we define ridge points for $f$. 
Ridge points play important role to study parallel surfaces.
\begin{dfn}\label{ridge}
Under the above settings, 
a point $p$ is called a \textit{ridge point} if $\bm{v}\kappa_+(p)=0$ holds, 
where $\bm{v}\kappa_+$ denotes the directional derivative of $\kappa_+$ with respect to $\bm{v}$. 
Moreover, a point $p$ is called a \textit{k-th order ridge point} if 
$\bm{v}^{(m)}\kappa_+(p)=0\ (1\leq m\leq k)$ and $\bm{v}^{(k+1)}\kappa_+(p)\neq0$ hold, 
where $\bm{v}^{(m)}\kappa_+$ means the $m$-th directional derivative of $\kappa_+$ with respect to $\bm{v}$.
\end{dfn}
Ridge points for regular surfaces were first studied deeply by Porteous \cite{p1}. 
He showed that ridge points correspond to $A_3$ singular points 
of distance squared functions on regular surfaces, 
that is, cuspidal edges of caustics. 
For more details on ridge points, see \cite{bgt,fh1,fh2,ifrt,p1,p2}.

\section{Parallel surfaces of wave fronts}\label{parallel}
For the case of regular surfaces, principal curvatures relate to singularities of parallel surfaces. 
In this section, 
we consider singularities of parallel surfaces of fronts and give criteria 
in terms of principal curvatures and other geometric properties. 
Swallowtails on parallel surfaces of cuspidal edges are studied in \cite{t1}. 
Here we give criteria for other singularities on parallel surfaces of fronts. 

\subsection{Singularities of parallel surfaces of wave fronts}
In this subsection, we shall deal with fronts which have singular points of the second kind
(swallowtails, for example). 
Needless to say, the following arguments can be applied to the case of cuspidal edges.  

Let $f:\Sig\to\R^3$ be a front, $\nu$ a unit normal to $f$ and 
$p\in \Sig$ a non-degenerate singular point of the second kind. 
Then the {\it paralle surface $f^t$ of $f$} is defined by 
$f^t=f+t\nu$, where $t\in \R\setminus\{0\}$ is constant. 
We note that $f^t$ is also a front since $\nu$ is a unit normal to $f^t$. 
\begin{lem}\label{singsetpara}
Let $f:\Sig\to\R^3$ be a front, $\nu$ its unit normal vector and $p$ a non-degenerate singular point of $f$. 
Suppose that $\kappa_+$ is a bounded $C^\infty$-function near $p$ and $\kappa_+(p)\neq0$. 
Then $p$ is a singular point of $f^t$ if and only if $t=1/\kappa_+(p)$. 
Moreover, $p$ is non-degenerate singular point of $f^t$ if and only if $p$ is not a critical point of $\kappa_+$.
\end{lem}
\begin{proof}
We show the case that $p$ is of the second kind. 
Let $(U;u,v)$ be an adapted coordinate system centered at $p$ with 
the null vector field $\eta=\partial_u+\epsilon(u)\partial_v$. 
Then the signed area density function $\lambda^t=\det(f^t_u,f^t_v,\nu)$ of $f^t$ can be written as 
\begin{equation*}
\lambda^t=\det(f^t_u,f^t_v,\nu)=(1-t\kappa_+)(\lambda-t\lambda\kappa_-)
\end{equation*}
by Lemma \ref{weinsec}, where $\lambda=\det(f_u,f_v,\nu)$. 
Since $\lambda\kappa_-$ does not vanish at $p$, 
$p$ is a singular point of $f^t$ if and only if $t=1/\kappa_+(p)$ holds. 
Thus we may treat $\what{\lambda}^t=\kappa_+(u,v)-\kappa_+(p)$ as the signed area density function of $f^t$. 
Non-degeneracy follows from $d\what{\lambda}^t=(\kappa_+)_udu+(\kappa_+)_vdv$. 
\end{proof}
\begin{thm}\label{sing_para}
Let $f:\Sig\to\bm{R}^3$ be a front and $p$ be a non-degenerate singular point. 
Suppose that the principal curvature $\kappa_+$ is a bounded $C^\infty$-function near $p$ and $\kappa_+(p)\neq0$. 
Then for the parallel surface $f^t$ with $t=1/\kappa_+(p)$, the following conditions hold.
\begin{enumerate}
\item[{\rm (1)}] 
Assume $d\kappa_+(p)\neq0$. 
Then the following hold$:$
\begin{itemize}
\item The map-germ $f^t$ at $p$ is $\mathcal{A}$-equivalent to a cuspidal edge 
if and only if $p$ is not a ridge point of $f$.
\item The map-germ $f^t$ at $p$ is $\mathcal{A}$-equivalent to a swallowtail if and only if 
$p$ is a first order ridge point of $f$. 
\item The map-germ $f^t$ at $p$ is $\mathcal{A}$-equivalent to a cuspidal butterfly if and only if 
$p$ is a second order ridge point of $f$.
\end{itemize}
\item[{\rm(2)}] 
Assume $d\kappa_+(p)=0$. 
Then the following hold$:$
\begin{itemize}
\item The map-germ $f^t$ at $p$ is $\mathcal{A}$-equivalent to a cuspidal lips if and only if 
$\rank\left(df^t\right)_p=1$ and $\det\mathcal{H}_{\kappa_+}(p)>0$ hold. 
\item The map-germ $f^t$ at $p$ is $\mathcal{A}$-equivalent to a cuspidal beaks if and only if 
$p$ is a first order ridge point of $f$, $\rank\left(df^t\right)_p=1$ 
and $\det\mathcal{H}_{\kappa_+}(p)<0$ hold.
\end{itemize}
Here $\mathcal{H}_{\kappa_+}(p)$ is the Hessian matrix of $\kappa_+$ at $p$. 
\end{enumerate}
\end{thm}
\begin{proof}
Let $f:\Sig\to\R^3$ be a front, $p\in \Sig$ a non-degenerate singular point of the second kind
 and $\nu$ a unit normal vector. 
Then we take an adapted coordinate system $(U;u,v)$ around $p$. 
By Lemma \ref{singsetpara}, 
we can take the signed area density function of the parallel surface $f^t$ with $t=1/\kappa_+(p)$ as 
$\what{\lambda}^t(u,v)=\kappa_+(u,v)-\kappa_+(p)$. 

First, we prove the assertion (1). 
In this case, $(\what{\lambda}^t)^{-1}(0)$ is a smooth curve near $p$ 
and there exists a null vector field $\eta^t$ of $f^t$. 
We set $\eta^t=\eta^t_1\partial_u+\eta^t_2\partial_v$, where $\eta^t_i\ (i=1,2)$ are functions on $U$. 
By Lemma \ref{weinsec}, $df^t(\eta^t)$ is written as 
\begin{multline*}
df^t(\eta^t)=\left[\left(v+t\frac{\what{F}(v\what{M}-\epsilon\what{N})}{\what{E}\what{G}-\what{F}^2}\right)\eta_1^t
+t\frac{\what{F}\what{N}-\what{G}\what{M}}{\what{E}\what{G}-\what{F}^2}\eta_2^t\right]\phi\\
+\left[\left(-\epsilon+
t\frac{\what{F}\what{L}-\what{E}(v\what{M}-\epsilon\what{N})}{\what{E}\what{G}-\what{F}^2}\right)\eta_1^t
+\left(1+t\frac{\what{F}\what{M}-\what{E}\what{N}}{\what{E}\what{G}-\what{F}^2}\right)\eta_2^t\right]f_v.
\end{multline*} 
Since $\phi$ and $f_v$ are linearly independent, $df^t(\eta^t)=\bm{0}$ on $S(f^t)$ is equivalent to 
\begin{equation*}
\begin{pmatrix}
\what{L}-\kappa_+(v\what{E}-\epsilon\what{F}) & \what{M}-\kappa_+\what{F}\\
v(\what{M}-\kappa_+\what{F})-\epsilon(\what{N}-\kappa_+\what{G}) & \what{N}-\kappa_+\what{G}
\end{pmatrix}
\begin{pmatrix}
\eta^t_1\\ \eta^t_2
\end{pmatrix}
=
\begin{pmatrix}
0\\0
\end{pmatrix}
\end{equation*}
holds on $S(f^{t})$. 
Thus the null vector field $\eta^t$ can be taken as the principal vector $\bm{v}$ 
as in \eqref{pdir_sec} with respect to $\kappa_+$ restricted to $S(f^{t})$. 
Under these conditions, 
the equation $(\eta^t)^{(k)}\what{\lambda}^{t}=\bm{v}^{(k)}\kappa_+$ holds for some natural number $k$. 
Thus we have the assertion (1) by Fact \ref{crit_front} (1). 

Next, we prove (2). 
In this case, $d\kappa_+$ vanishes at $p$. 
We consider the rank of $df^t$ at $p$. 
The Jacobian matrix $J_{f^t}$ of $f^t$ is $J_{f^t}=(\phi,f_v)\mathcal{M}$ at $p$, where 
\begin{equation}
\mathcal{M}=
\begin{pmatrix}
0 & 0\\
0 & 1
\end{pmatrix}-t
\begin{pmatrix}
\what{E} & \what{F}\\
\what{F} & \what{G}
\end{pmatrix}^{-1}
\begin{pmatrix}
\what{L} & \what{M}\\
0 & \what{N}
\end{pmatrix}=
\frac{1}{\what{N}(\what{E}\what{G}-\what{F}^2)}
\begin{pmatrix}
{-\what{G}^2\what{L}} & \what{G}(\what{F}\what{N}-\what{G}\what{M})\\
\what{F}\what{G}\what{L} & -\what{F}(\what{F}\what{N}-\what{G}\what{M})
\end{pmatrix}
.\label{calm}
\end{equation} 
Since $\rank\mathcal{M}=1$, 
it follows that 
$\rank\left(J_{f^t}\right)_p=1$, when $t=1/\kappa_+(p)$, 
and it implies that $\rank\left(df^t\right)_p=1$. 
Thus there exists a non-zero vector field $\eta^t$ near $p$ 
such that if $q\in S(f^t)$ then $df^t(\eta^t)=\bm{0}$ holds at $q$. 
We can take the principal vector $\bm{v}$ with respect to $\kappa_+$ as $\eta^t$, 
then $\eta^t\eta^t\what{\lambda}^{t}=\bm{v}^{(2)}\kappa_+$. 
Moreover, we see that 
$\what{\lambda}^{t}_{uu}=(\kappa_+)_{uu}$, $\what{\lambda}^{t}_{uv}=(\kappa_+)_{uv}$, $\what{\lambda}^{t}_{vv}=(\kappa_+)_{vv}$. 
Thus we have 
$\det\mathcal{H}_{\what{\lambda}^{t}}(p)=\det\mathcal{H}_{\kappa_+}(p)$.  
By using Fact \ref{crit_front} (2) 
and the definition of ridge points, we have the conclusion. 
\end{proof}
This theorem implies that the behavior of a bounded principal curvature of fronts 
determines the types of singularities 
appearing on parallel surfaces. 
For regular surfaces and Whitney umbrellas, similar results are obtained in \cite{fh1,fh2}. 
By \eqref{calm} in the proof of Theorem \ref{sing_para} and \cite[Theorem 1.1]{s1}, we see that 
a parallel surface $f^t$ does not have $D_4$ singularity at $p$.

\subsection{Constant principal curvature lines and exactly cusped points of cuspidal edges}\label{cpc}
Let $f:\Sig\to\R^3$ be a front, $\nu$ a unit normal vector and $p$ a cuspidal edge. 
Suppose that $\kappa_+$ is bounded at $p$ and $\kappa_+(p)\neq0$. 
We set $\what{\lambda}^t(u,v)=\kappa_+(u,v)-\kappa_+(p)$. 
The zero-set of this function gives the singular curve of the parallel surface $f^t$ of $f$, 
where $t=1/\kappa_+(p)$ (Lemma \ref{singsetpara}). 
We call the curve given by $\what{\lambda}^t(u,v)=\kappa_+(u,v)-\kappa_+(p)=0$ a 
\textit{constant principal curvature $($CPC\/$)$ line with the value of $\kappa_+(p)$} (cf. \cite{fh1,fh2}). 
In this case, the CPC line is a regular curve since $d\what{\lambda}^t(p)\neq0$. 
In \cite{fh1,fh2}, 
CPC lines for regular surfaces and Whitney umbrellas, 
and relations between singularities of parallel surfaces and the behavior of CPC lines are investigated. 
For intrinsic properties of Whitney umbrellas, see \cite{hhnsuy,hhnuy}. 

First, we consider contact of the CPC line with the singular curve. 
\begin{dfn}
Let $\alpha:I\ni t\mapsto(x(t),y(t))\in\R^2$ be a regular plane curve and let $\beta$ be another plane curve 
given as the zero set of a $C^\infty$-function $F:\R^2\to\R$, where $I\subset\R$ is an open interval. 
Then $\alpha$ has {\it $(n+1)$-point contact} at $t_0\in I$ with $\beta$ if 
the function $g(t)=F\circ\alpha(t)=F(x(t),y(t))$ satisfies 
$$g(t_0)=g'(t_0)=g''(t_0)=\cdots=g^{(n)}(t_0)=0\quad\textrm{and}\quad g^{(n+1)}(t_0)\neq0,$$
where $'=d/dt$ and $g^{(m)}$ denotes the $m$-th order derivative of $g$.
\end{dfn}

\begin{prop}\label{contactness}
Let $f:\Sig\to\R^3$ be a front, $p$ a cuspidal edge 
and $\gamma$ a singular curve passing through $p$. 
Suppose that $\kappa_+$ is bounded near $p$ and $d\kappa_+(p)\neq0$. 
Then $\gamma$ has $(n+1)$-point contact at $p$ with the CPC line if and only if 
$$\kappa_\nu'(p)=\cdots=\kappa_\nu^{(n)}(p)=0\quad\text{and}\quad \kappa_\nu^{(n+1)}(p)\neq0.$$
\end{prop}
\begin{proof}
Let $(U;u,v)$ be an adapted coordinate system. 
Then $\kappa_+(u,0)=\kappa_\nu(u)$ holds by Theorem \ref{princi1}.
Thus the composite function of $\what{\lambda}^t$ and $\gamma$ is given as 
$$\what{\lambda}^t(u,0)=\kappa_\nu(u)-\kappa_\nu(p)$$
since $\kappa_+(p)=\kappa_\nu(p)$. 
Hence we get the conclusion by the definition of contact of two plane curves.
\end{proof}

Next, we consider special points (landmarks in the sense of Porteous \cite{p2}) on CPC lines of cuspidal edges. 
In this case, we use the following normal form obtained by Martins and Saji \cite{ms}.
\begin{fact}[{\cite[Theorem 3.1]{ms}}]\label{normal_cusp}
Let $f:(\bm{R}^2,\bm{0})\to (\bm{R}^3,\bm{0})$ be a map-germ and $\bm{0}$ a cuspidal edge. 
Then there exist a diffeomorphism-germ $\theta:(\bm{R}^2,\bm{0})\to (\bm{R}^2,\bm{0})$ 
and an isometry-germ $\Theta:(\bm{R}^3,\bm{0})\to (\bm{R}^3,\bm{0})$ satisfying 
\begin{equation}\label{normal}
\Theta\circ f \circ\theta(u,v)=
\left( u,\frac{a_{20}}{2}u^2+\frac{a_{30}}{6}u^3+\frac{v^2}{2},
\frac{b_{20}}{2}u^2+\frac{b_{30}}{6}u^3+\frac{b_{12}}{2}uv^2+\frac{b_{03}}{6}v^3\right)+h(u,v),
\end{equation}
where $b_{20}\geq0,\,b_{03}\neq0$ and 
\begin{equation*}
h(u,v)=(0,u^4h_1(u),u^4h_2(u)+u^2v^2h_3(u)+uv^3h_4(u)+v^4h_5(u,v)),
\end{equation*}
with $h_i(u)\ (1\leq i \leq 4),\ h_5(u,v)$ smooth functions.
\end{fact}
We note that coefficients in the normal form \eqref{normal} correspond to 
$\kappa_s(0)=a_{20}$, $\kappa_\nu(0)=b_{20}$, $\kappa_t(0)=b_{12}$ and $\kappa_c(0)=b_{03}$ (see \cite{ms}).

Let $f:\Sig\to\R^3$ be a front, $p\in \Sig$ a cuspidal edge and assume that $\kappa_+$ is bounded near $p$. 
The condition $\eta\kappa_+=0$ at $p$ implies that the CPC line is tangent to the null vector field $\eta$ of $f$ at $p$. 
Moreover, the image $f(S(f^t))$ of the set of singular points of the parallel surface $f^t$ by $f$ is cusped at $p$. 
We call such a point an \textit{exactly cusped point for the constant principal curvature $($CPC\/$)$ line}. 
\begin{prop}\label{negative}
Let $f:\Sig\to\R^3$ be a front and $p$ a cuspidal edge. 
Suppose that $\kappa_+$ $($resp. $\kappa_-)$ is bounded at $p$. 
Then 
$\eta\kappa_+(p)=0$ $($resp. $\eta\kappa_-(p)=0)$ 
implies $\kappa_s(p)\leq0$.
\end{prop}
\begin{proof}
By using a normal form as in \eqref{normal}, we have 
\begin{equation}
(\kappa_+)_u=b_{30}-a_{20}b_{12},\quad 
 (\kappa_+)_v=-(4b_{12}^2+a_{20}b_{03}^2)/2b_{03}\label{diffk+}
\end{equation}
at $\bm{0}$ (see \cite[Lemma 2.2]{t1}). 
Since the null vector is $\eta=\partial_v$ for a normal form \eqref{normal}, 
the relation $(\kappa_+)_v=\eta\kappa_+$ holds. 
Hence $\eta\kappa_+(\bm{0})=0$ if and only if 
$4b_{12}^2+a_{20}b_{03}^2=0$. 
This implies that 
$$\kappa_s(0)=a_{20}=-\frac{4b_{12}^2}{b_{03}^2}\leq0.$$
Thus we obtain the assertion.
For the case of $\kappa_-$ to be bounded, we can show in a similar way.
\end{proof}

Relations between the Gaussian curvature and 
the singular curvature are stated in \cite[Theorem 3.1]{suy}. 

\begin{prop}\label{exactcusp}
Let $f:\Sig\to\R^3$ be a front, $p$ a cuspidal edge, $\gamma$ a singular curve and $\eta$ a null vector field. 
Assume that $\kappa_+$ is bounded near $p$, $\kappa_+(p)\neq0$ 
and $p$ is not a ridge point of $f$. 
Then the cusp-directional torsion $\kappa_t^t$ of $f^t$ vanishes at $p$ if and only if $\eta\kappa_+$ vanishes at $p$, 
namely, $p$ is an exactly cusped point, 
where $t=1/\kappa_+(p)$. 
\end{prop}
\begin{proof}
Let $f$ be a normal form as in \eqref{normal} and $\sigma$ be a singular curve of $f^t$ satisfting $\what{\lambda}^t(\sigma)=0$. 
Note that coefficients in \eqref{normal} satisfy $b_{20}\neq0$ and 
$4b_{12}^3+b_{30}b_{03}^2\neq0$ since $\kappa_+(\bm{0})=\kappa_\nu(0)=b_{20}$ and 
$\bm{0}$ is not a ridge point (see \cite[Lemma 2.2]{t1}). 
We assume that $(\kappa_+)_u(\bm{0})\neq0$. 
Then we can take $\sigma(v)=(u(v),v)$. 
Let $\bm{w}=u'\partial_u+\partial_v$ denote a vector field tangent to $\sigma$, 
where $u'=-(\kappa_+)_v/(\kappa_+)_u$. 
The pair $(\bm{w},\bm{v})$ gives an adapted pair of vector fields in the sense of \cite{ms}. 
Moreover, $\langle\bm{w}f^t,\bm{v}\bm{v}f^t\rangle=0$ holds at $\bm{0}$. 
By \cite[(5.1)]{ms}, we have
\begin{equation}
\kappa_t^t(0)=\frac{\det(\bm{w}f^t,\bm{v}\bm{v}f^t,\bm{w}\bm{v}\bm{v}f^t)}
{\|\bm{w}f^t\times\bm{v}\bm{v}f^t\|^2}(\bm{0})
=\frac{b_{20}^2(4b_{12}^2+a_{20}b_{03}^2)}{4b_{12}^3+b_{30}b_{03}^2}.\label{ktt}
\end{equation}
Comparing \eqref{diffk+} and \eqref{ktt}, we obtain the result.
\end{proof}
We now consider the case that $(\kappa_+)_u=0$ at $p$. 
Since this is equivalent to $\kappa_\nu'=0$ at $p$, 
we call such a point an \textit{extrema of the limiting normal curvature} $\kappa_\nu$. 
Therefore we have three special points on cuspidal edges which have 
special relations between the singular curve and the CPC line (see Figure \ref{fig:cpc}). 
It seems that exactly cusped points have not appeared in the literature. 

\begin{figure}[htbp]
  \begin{center}
    \begin{tabular}{c}

      \begin{minipage}{0.3\hsize}
        \begin{center}
          \includegraphics[clip, width=3.5cm]{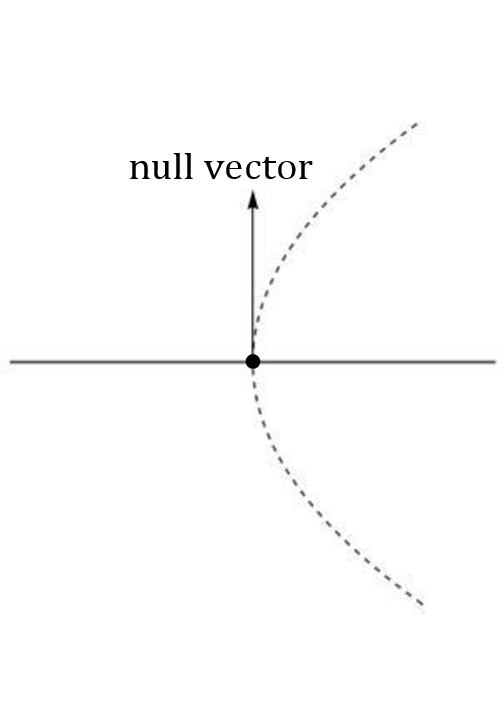}
          \hspace{0.5cm} exactly cusped point
        \end{center}
      \end{minipage}

      \begin{minipage}{0.3\hsize}
        \begin{center}
          \includegraphics[clip, width=3.5cm]{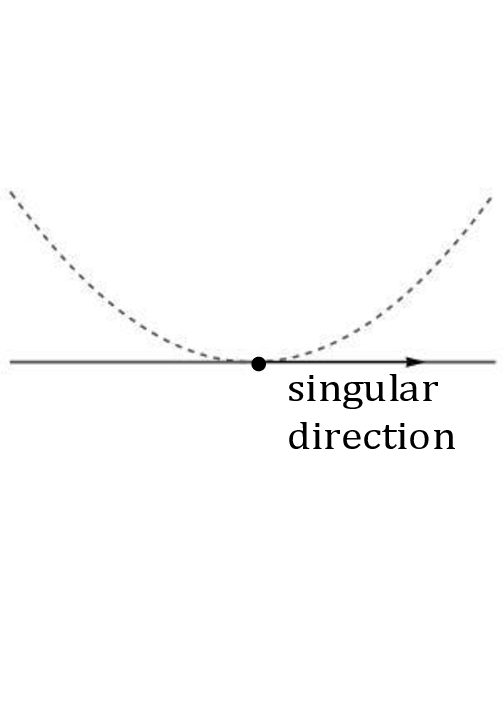}
          \hspace{0.5cm} extrema of $\kappa_\nu$
        \end{center}
      \end{minipage}

      \begin{minipage}{0.3\hsize}
        \begin{center}
          \includegraphics[clip, width=3.5cm]{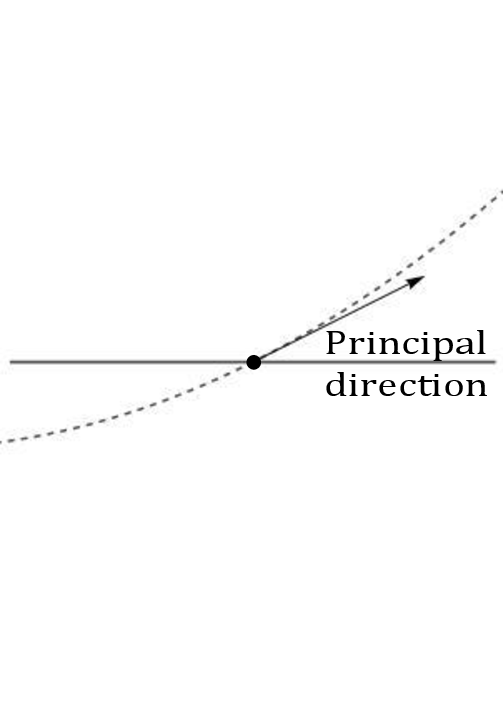}
          \hspace{0.5cm} ridge point
        \end{center}
      \end{minipage}

    \end{tabular}
    \caption{Figures of the singular curve and the CPC line near a cuspidal edge. 
    The solid curve is the singular curve and the dotted one is the CPC line through $p$.}
    \label{fig:cpc}
  \end{center}
\end{figure}

\section{Extended distance squared functions on wave fronts}\label{distsqared}
We consider the extended distance squared functions on fronts. 
We study relations between singularities of extended distance squared functions and principal curvatures. 
For singularities of distance squared functions on surfaces with other corank one singularities, see \cite{fh2,mn}. 

Let $f:\Sig\to\R^3$ be a front, $\nu$ a unit normal vector and $p$ be a non-degenerate singular point of the second kind. 
(For cuspidal edges, see \cite{t1}.) 
Assume that $\kappa_+$ is bounded at $p$ and $\kappa_+(p)\neq0$ in this section. 

We set the function $\psi:\Sig\to\R$ as
\begin{equation}\label{dist2}
\psi(u,v)=-\frac{1}{2}(\|\bm{x}_0-f(u,v)\|^2-t_0^2),
\end{equation} 
where $\bm{x}_0\in\R^3$ and $t_0\in\R\setminus\{0\}$. 
We call $\psi$ the {\it extended distance squared function with respect to $\bm{x}_0$}. 
\begin{lem}\label{singpsi}
For the function $\psi$ as in \eqref{dist2}, $\psi(p)=\psi_u(p)=\psi_v(p)=0$ if $\bm{x}_0=f(p)+t_0\nu(p)$. 
\end{lem}
\begin{proof}
Let us take an adapted coordinate system $(U;u,v)$ centered at $p$ with 
the null vector field $\eta=\partial_u+\epsilon(u)\partial_v$. 
In this case, $\psi(p)=0$ follows from \eqref{dist2}. 
By direct computations, we have $\psi_u=\langle{\bm{x}_0-f,v\phi-\epsilon f_v}\rangle$ 
and $\psi_v=\langle{\bm{x}_0-f,f_v}\rangle$. 
Since $\langle{\phi,\nu}\rangle=\langle{f_v,\nu}\rangle=0$ at $p$, we get the conclusion.
\end{proof}
Take $\bm{x}_0=f(p)+t_0\nu(p)$. 
We are interested in the case of $t_0=1/\kappa_+(p)$, because 
$\bm{x}_0$ corresponds to the image of a singular point of a parallel surface $f^t$ with $t=1/\kappa_+(p)$, 
that is, $\bm{x}_0$ coincides with a {\it focal point} of $f$ at $p$. 
In such a case, $\psi$ measures contact of $f$ with the {\it principal curvature sphere} at $p$ (cf. \cite{ifrt,mn}). 
\begin{prop}\label{rankzero}
If $\bm{x}_0=f(p)+\nu(p)/\kappa_+(p)$ and $t_0=1/\kappa_+(p)$, 
then $j^2\psi=0$ holds, where $j^2\psi$ is the $2$-jet of $\psi$ at $p$.
\end{prop}
\begin{proof}
Take an adapted coordinate system $(U;u,v)$ around $p$. 
By Lemma \ref{singpsi}, we see that $\psi=\psi_u=\psi_v=0$ at $p$. 
By direct calculations, we have 
\begin{align*}
\psi_{uu}&=-\|v\phi-\epsilon f_v\|^2+\langle\bm{x}_0-f,v\phi_u-\epsilon' f_v-\epsilon f_{uv}\rangle,\\
\psi_{uv}&=-\langle f_v,v\phi-\epsilon f_v\rangle+\langle\bm{x}_0-f,\phi+v\phi_v-\epsilon f_{vv}\rangle,\\
\psi_{vv}&=-\| f_v\|^2+\langle\bm{x}_0-f,f_{vv}\rangle.
\end{align*}
Thus $\psi_{uu}=\psi_{uv}=0$ hold at $p$ since $\langle{f_v,\nu}\rangle=\langle{\phi,\nu}\rangle=0$. 
Moreover, it follows that 
$\psi_{vv}=-\what{G}+\langle{\nu,f_{vv}}\rangle/\kappa_+(p)=-\what{G}+\what{N}/\kappa_+(p)=0$ 
at $p$ since $1/\kappa_+(p)=\what{G}(p)/\what{N}(p)$. 
Thus we have the assertion. 
\end{proof}
We note that Martins and Nu{\~ n}o-Ballesteros \cite{mn} investigate singularities of distance squared functions 
in more general situation. 
They showed a similar result as Proposition \ref{rankzero} by using an 
{\it umbilic curvature} $\kappa_u$ \cite[Theorem 2.15]{mn}. 
It is known that $|\kappa_\nu(p)|=\kappa_u(p)$ holds when $p$ is acuspidal edge (\cite{ms}). 
Thus Proposition \ref{rankzero} might be a special case of \cite[Theorem 2.15]{mn}.

Proposition \ref{rankzero} implies that $\psi$ may have a $D_4$ singularity at $p$ 
if $\bm{x}_0$ coincides with the focal point of $f$ at $p$, where a function-germ 
$h:(\R^2,\bm{0})\to(\R,0)$ has a {\it $D_4$ singularity} at $\bm{0}$ if $h$ is $\mathcal{R}$-equivalent to the germ 
$(u,v)\mapsto u^{3}\pm uv^2$ at $\bm{0}$ (cf. \cite[pages 264 and 265]{bg}). 
Therefore we consider the condition that $\psi$ has a $D_4$ singularity at $p$ in terms of geometric properties of $f$. 

For cuspidal edges, 
suppose that $\kappa_+$ is bounded near $p$ and $\kappa_+(p)\neq0$. 
Then the function $\psi$ with $\bm{x}_0=f(p)+\nu(p)/\kappa_+(p)$ and $t_0=1/\kappa_+(p)$ 
has a $D_4$ singularity if and only if 
$\kappa_i(p)(4\kappa_t(p)^3+\kappa_i(p)\kappa_c(p)^2)=0$ (\cite[Theorem 3.3]{t1}). 
Here $\kappa_i$ is a geometric invariant called the {\it edge inflectional curvature} defined in \cite[Section 5.3]{ms}. 
We remark that Oset Sinha and Tari \cite{of2} study singularities of 
height functions and orthogonal projections of cuspidal edges. 

Let $f:\Sig\to\R^3$ be a front and $p$ a sngular point of the second kind. 
For a function $\psi:V\to\R$, set 
\begin{multline}\label{delta}
\Delta_{\psi}=((\psi_{uuu})^2(\psi_{vvv})^2-6\psi_{uuu}\psi_{uuv}\psi_{uvv}\psi_{vvv}\\
-3(\psi_{uuv})^2(\psi_{uvv})^2+4(\psi_{uuv})^3\psi_{vvv}+4\psi_{uuu}(\psi_{uvv})^3)(p).
\end{multline}
It is known that the function $\psi$ is $\mathcal{R}$-equivalent to $u^3+uv^2$ (resp. $u^3-uv^2$) 
if and only if $j^2\psi=0$ and $\Delta_{\psi}>0$ (resp. $\Delta_{\psi}<0$) hold
(see \cite[Lemma 3.1]{s1}, see also \cite[Theorem 4.2]{fh1}). 

\begin{thm}\label{D4front}
Let $f:\Sig\to\R^3$ be a front and $p$ a singular point of the second kind. 
Suppose that $\kappa_+$ is bounded near $p$ and $\kappa_+(p)\neq0$. 
Then $\psi$ as in \eqref{dist2} 
with $\bm{x}_0=f(p)+\nu(p)/\kappa_+(p)$ and $t_0=1/\kappa_+(p)$ has a $D_4$ singularity at $p$ 
if and only if $p$ is not a ridge point of $f$. 
\end{thm}

To prove this theorem, 
we take a special adapted coordinate system $(U;u,v)$ centered at $p$ 
called a {\it strongly adapted coordinate system} 
which satisfies $\langle{f_{uv},f_v}\rangle=0$ at $p$ (see \cite[Definition 3.6]{msuy}). 
Under this coordinate system, we see that $\what{F}=\what{G}_u=0$ at $p$ since $\phi(p)=f_{uv}(p)$. 
We prepare a lemma. 
\begin{lem}\label{del}
Under the above conditions, 
$\Delta_{\psi}\neq0$ 
if and only if 
\begin{equation}\label{d4psi}
4\psi_{uuv}\psi_{vvv}-3(\psi_{uvv})^2=
\frac{4\what{G}}{\what{N}^2}(\what{L}(\what{G}\what{N}_v-\what{G}_v\what{N})-
\what{G}\what{M}(\what{N}_u+\what{M}))\neq0
\end{equation}
at $p$.
\end{lem}
\begin{proof}
We take a strongly adapted coordinate system $(U;u,v)$ around $p$. 
Direct calculations show that 
$$\psi_{uuu}=t_0\langle\nu,f_{uuu}\rangle,\quad
\psi_{uuv}=t_0\langle\nu,f_{uuv}\rangle-\langle f_{v},f_{uu}\rangle$$
hold at $p$, where $t_0=1/\kappa_+(p)=\what{G}(p)/\what{N}(p)$. 
Since $f_{uu}=-\epsilon' f_v$, $f_{uuu}=-\epsilon''f_v-2\epsilon'\phi$ and $f_{uuv}=\phi_u-\epsilon'f_{vv}$ at $p$, 
it follows that $\psi_{uuu}=0$ and 
\begin{equation}\label{psi_uuv}
\psi_{uuv}=
t_0\langle{\nu,\phi_u}\rangle+\epsilon'(-t_0\langle\nu,f_{vv}\rangle+\|f_v\|^2)=t_0\langle{\nu,\phi_u}\rangle=
\what{G}\what{L}/\what{N}\neq0
\end{equation}
hold at $p$. 
Thus $\Delta_{\psi}$ as in \eqref{delta} can be written as 
$$\Delta_{\psi}=(\psi_{uuv}(p))^2(4\psi_{uuv}(p)\psi_{vvv}(p)-3(\psi_{uvv}(p))^2).$$ 
This implies that $\Delta_{\psi}\neq0$ if and only if $4\psi_{uuv}(p)\psi_{vvv}(p)-3(\psi_{uvv}(p))^2\neq0$. 

We consider the form of $4\psi_{uuv}(p)\psi_{vvv}(p)-3(\psi_{uvv}(p))^2\neq0$. 
By direct computations, we have 
$$\psi_{uvv}=t_0\langle{\nu,f_{uvv}}\rangle,\quad \psi_{vvv}=t_0\langle{\nu,f_{vvv}}\rangle-3\langle{f_v,f_{vv}}\rangle$$
at $p$. 
Since $f_{uvv}=2\phi_v$ at $p$, it follows that 
\begin{equation}\label{psi_uvv}
\psi_{uvv}(p)=2t_0\what{M}(p)=\frac{2\what{G}(p)\what{M}(p)}{\what{N}(p)}.
\end{equation} 

We now deal with $\psi_{vvv}(p)$. 
It follows that $\langle{\nu,f_{v}}\rangle=0$ and 
$\langle{\nu,f_{vv}}\rangle=-\langle{\nu_v,f_v}\rangle=\what{N}$ on $U$. 
So $\langle{\nu,f_{vvv}}\rangle=\what{N}_v-\langle{\nu_v,f_{vv}}\rangle$ holds. 
By Lemma \ref{weinsec}, $\langle{\nu_v,f_{vv}}\rangle$ is written as 
$$\langle{\nu_v,f_{vv}}\rangle=
-\frac{\what{M}}{\what{E}}\langle{\phi,f_{vv}}\rangle-\frac{\what{N}}{\what{G}}\langle{f_v,f_{vv}}\rangle$$
at $p$. 
On the other hand, 
$\what{N}_u=\langle{\nu_u,f_{vv}}\rangle+\langle{\nu,f_{uvv}}\rangle=-\what{L}\langle{\phi,f_{vv}}\rangle/\what{E}+2\what{M}$ 
at $p$ by Lemma \ref{weinsec}. 
Hence we have 
$\langle{\phi,f_{vv}}\rangle=-\what{E}(\what{N}_u-2\what{M})/\what{L}$ and
\begin{equation}\label{psi_vvv}
\psi_{vvv}=\frac{\what{G}\what{N}_v-\what{G}_v\what{N}}{\what{N}}-
\frac{\what{G}\what{M}(\what{N}_u-2\what{M})}{\what{L}\what{N}}
\end{equation}
at $p$, where we used $2\langle{f_v,f_{vv}}\rangle=\what{G}_v$. 
By \eqref{psi_uuv}, \eqref{psi_uvv} and \eqref{psi_vvv}, $4\psi_{uuv}\psi_{vvv}-3(\psi_{uvv})^2$ can be written as 
\begin{equation*}
4\psi_{uuv}\psi_{vvv}-3(\psi_{uvv})^2=
\frac{4\what{G}}{\what{N}^2}(\what{L}(\what{G}\what{N}_v-\what{G}_v\what{N})-\what{G}\what{M}(\what{N}_u+\what{M}))
\end{equation*}
at $p$. Thus we have the assertion.
\end{proof}

\begin{proof}[Proof of Theorerm \ref{D4front}]
Let us take a strongly adapted coordinate system $(U;u,v)$ centered at $p$. 
Then we note that $\what{F}=\what{G}_u=0$ holds at $p$. 
The differentials $(\kappa_+)_u$ and $(\kappa_+)_v$ are given by 
$$(\kappa_+)_u=\frac{\what{N}_u}{\what{G}},\quad
(\kappa_+)_v=\frac{-\what{G}\what{M}^2+\what{L} (\what{G}\what{N}_v-\what{G}_v\what{N})}
{\what{G}^2\what{L}}$$
at $p$. 
Since the principal vector $\bm{v}$ as in \eqref{pdir_sec} is written as $\bm{v}=(-\what{M},\what{L})$ at $p$, 
we have 
\begin{align}\label{strongridge}
\bm{v}\kappa_+(p)&=-\what{M}(p)(\kappa_+)_u(p)+\what{L}(p)(\kappa_+)_v(p)\\
&=\frac{1}{\what{G}(p)^2}(\what{L}(p) (\what{G}(p)\what{N}_v(p)-\what{G}_v(p)\what{N}(p))-
\what{G}(p)\what{M}(p)(\what{N}_u(p)+\what{M}(p))).\nonumber
\end{align}
Comparing \eqref{strongridge} and \eqref{d4psi} in Lemma \ref{del}, 
$\bm{v}\kappa_+(p)\neq0$, namely, $p$ is not a ridge point of $f$ 
if and only if $4\psi_{uuv}(p)\psi_{vvv}(p)-3(\psi_{uvv}(p))^2\neq0$. 
This implies that the number $\Delta_{\psi}$ defined as \eqref{delta} does not vanish by Lemma \ref{del}. 
\end{proof}
We remark that the condition that $f$ is a front in Theorem \ref{D4front} is needed 
for $\psi$ to have a $D_4$ singularity at $p$. 
In fact, for a frontal $f:\Sig\to\R^3$ with a singular point of the admissible second kind $p$, we have the following.
\begin{prop}\label{D5frontal}
Let $f:\Sig\to\R^3$ be a frontal but not a front and $p$ a singular point of the admissible second kind. 
Then $\psi$ with $\bm{x}_0=f(p)+\nu(p)/\kappa_\nu(p)$ and $t_0=1/\kappa_\nu(p)$ 
does not have a $D_4$ singularity at $p$. 
\end{prop}
\begin{proof}
Let us take an adapted coordinate system $(U;u,v)$ centered at $p$ with the null vector field 
$\eta=\partial_u+\epsilon(u)\partial_v$. 
Since $f$ at $p$ is a frontal but not a front, $\what{L}(p)=0$. 
Thus $\psi_{uuv}(p)=0$ by \eqref{psi_uuv}. 
By the proof of Lemma \ref{del}, $\Delta_\psi$ vanishes automatically if $f$ is not a front at $p$. 
\end{proof}

\proof[Acknowledgements] 
The author thanks Professor Kentaro Saji 
for fruitful discussions and valuable comments. 
He also thanks the referee for reading the manuscript carefully and helpful suggestions.

\end{document}